\documentclass{amsart}

\usepackage{amsfonts,amsmath,amsthm,amssymb}

\newtheorem{theorem}{Theorem}
\newtheorem{lemma}{Lemma}
\newtheorem{corollary}{Corollary}

\def\gp#1{\langle #1 \rangle}
\def\m1{^{-1}}

\title[regularity]
{On the regularity of crossed products}
\author{V.~Bovdi, S.~Mihovski}

\address{
Department of Math. Sciences,
UAE University - Al-Ain,
United Arab Emirates}
\email{vbovdi@gmail.com}

\address{
Department of Algebra, University of Plovdiv, Bulgaria}
\email{mihovski@uni-plovdiv.bg}


\begin{document}
\begin{abstract}
We study some generalizations of the notion of regular crossed products $K*G$. For the case when $K$  is an algebraically closed field, we give necessary and sufficient conditions for the twisted group ring $K* G$ to be an $n$-weakly regular ring, a $\xi^* N$-ring or a ring without nilpotent elements.
\end{abstract}

\maketitle

\section{Introduction}

Let  $G$ be a group, $U(K)$ the  group of units of the associative ring $K$ with identity and let  $\sigma: G\to \mathrm{Aut}(K)$ be  a map of   $G$ into the group $\mathrm{Aut}(K)$ of automorphisms  of $K$.
Let $K*G=K_\rho^\sigma G=\{\sum_{g\in G}u_g\alpha_g\mid \alpha_g\in K\}$ be the crossed product (in the sense of \cite{Bovdi}),  of  the group $G$ over the ring $K$ with  respect to the factor system
\[
\rho=\{\rho(g,h)\in U(K)\mid g,h\in G\}
\]
and the map $\sigma:G\to \mathrm{Aut}(K)$. Moreover we assume that the factor system $\rho$ is normalized, i.e.
$\rho(g,1)=\rho(1,g)=\rho(1,1)=1$ for any $g\in G$.

In particular, if $\sigma=1$, then the crossed product $K*G$ is called  a {\it twisted group ring}, which we  denote by $K_\rho G$. If the factor system $\rho$ is unitary, i.e.  $\rho(g,h)=1$ for all $g,h\in G$, then $K*G$ is called a {\it skew group ring} and is denoted by $K^{\sigma}G$. In the case, when $\rho=1$ and $\sigma=1$, then $K*G$ is the ordinary group ring $KG$.

In the present paper we study properties of crossed products $K*G$
which are generalizations of the notion of a regular ring. For the
case when $K*G$  is a  twisted group ring over   the algebraically
closed field $K$, we give necessary and sufficient conditions  for
$K*G$ to be an $n$-weakly regular ring ($n\geq 2$), a $\xi^*
N$-ring or a ring without nilpotent elements.  Our investigation
can be considered as a generalization of  certain results of
\cite{Bovdi_Langi, Connell, Gupta, Mihovski, Rakhnev, Utumi}
earlier obtained for group rings. Note that we exclude  the case
when $K*G$ is a skew group ring, so we do not cite any reference
from that topic.

\section{Twisted group algebras without nilpotent elements}

Denote  the $K$-basis of $K*G$ by $U_G=\{u_g \mid  g\in G\}$. The multiplication of $u_g, u_h\in U_G$ is defined by $u_gu_h=\rho(g,h)u_{gh}$, where  $\rho(g,h)\in \rho$ and  $g,h\in G$. The factor system $\rho$ of the crossed product $K*G$ is called {\it symmetric},
if for all elements $g,h\in G$ the condition $gh=hg$ yields  $\rho(g,h)=\rho(h,g)$.
The finite subset
$\mathrm{Supp}(a)=\{g\in G\mid \alpha_g\not=0\}$
of $G$ is called the  {\it support} of the element $a\in K*G$.

We shall  freely use the following.
\begin{lemma}\label{L:1}
Let  $K*G$ be   a crossed product and suppose that $axb=c$ for
some  $x,a,b,c\in K*G$. If $H$ is the subgroup of $G$ generated by
$\mathrm{Supp}(a)$, $\mathrm{Supp}(b)$ and $\mathrm{Supp}(c)$,
then there exists an element $y\in K*H$, such that $ayb=c$.
\end{lemma}
\begin{proof}
Indeed, if $x=y+z$, then $ayb+azb=c$, where  $y=\sum_{h\in H}u_h\alpha_h$ and  $z=\sum_{g\not\in H}u_g\beta_g$.
This shows that $\mathrm{Supp}(azb)\subseteq H$. Since $fgh\not\in H$ for $f\in  \mathrm{Supp}(a)$, $g\in  \mathrm{Supp}(z)$ and $h\in  \mathrm{Supp}(b)$, we conclude that $azb=0$ and $ayb=c$, as it was requested.
\end{proof}

\begin{corollary}\label{C:1}
If $g\in G$ has  infinite order, then $u_g-1$ is neither a one-sided zero divisor, nor a one-sided invertible element of the crossed product $K*G$.
\end{corollary}

\begin{proof}
In fact, if $u_g-1$ is either a one-sided zero divisor, or a one-sided invertible element of $K*G$, then by Lemma \ref{L:1}, we may assume that $u_g-u_1$ is also   such an element of $K*H$, where $H=\gp{g}$ is an infinite cyclic group. But $H$ is an ordered group, a contradiction.
\end{proof}

For twisted group algebras we give a refinement of Corollary 2 and
Lemma 2 of \cite{Mihovski} (see p.68) which were earlier proved
for group rings.

\begin{theorem}\label{T:1}
Let $K_\rho G$ be a twisted group algebra  of a torsion group $G$ over the algebraically closed
field $K$. The ring  $K_\rho G$ does not contain nilpotent elements if
and only if the following conditions hold:
\begin{itemize}
\item[(i)] $G$ is an abelian group;
\item[(ii)] the order of every elements in $G$ is  invertible in $K$;
\item[(iii)]  the factor system $\rho$ is symmetric.
\end{itemize}
\end{theorem}
\begin{proof}
Assume that the conditions  (i), (ii) and (iii) hold. Then the twisted group ring $K_\rho G$ is commutative. If $x\in K_\rho G$ is a nonzero nilpotent element and $H=\gp{ \mathrm{Supp}(x)}$, we conclude that $K_\rho H$ is a commutative  artinian ring with a nonzero nilpotent element  $x$. So, by Theorem 2.2 of (\cite{Passman_radical_I}, p.415), we get a contradiction.

Conversely, let $K_\rho G$ be a twisted group ring without nilpotent elements. If $g\in G$ is of order $n$ and
$u_g^n=u_1\alpha_g$, where  $\alpha_g\in U(K)$, then there exists an element $\mu_g\in U(K)$ such that
$\mu_g^n=\alpha_g^{-1}$, because $K$ is algebraically closed. So for the element $v_g=u_g\mu_g$ we
have  $v_g^n=1$. Obviously,
\[
x=(v_g-1)u_h(1+v_g+v_g^2+\cdots+v_g^{n-1})
\]
is a nilpotent element of    $K_\rho G$ for all $h\in G$ as far as $x^2=0$. Thus $x=0$, so we conclude that
\begin{equation}\label{E:1}
u_h=v_gu_hv_g^i\quad\qquad (0\leq i\leq n-1).
\end{equation}
Examining  the supports  we can deduce that $h\m1gh=g^{-i}$ \quad ($h\in G$). Therefore  all cyclic subgroups of $G$ are normal. This implies that $G$ is either abelian or hamiltonian. If $gh=hg$, then $i=n-1$ and by (\ref{E:1}) it follows that $u_hv_g=v_gu_h$, since $v_g^n=1$ is the identity element of $K_\rho G$. So we conclude that $\rho(g,h)=\rho(h,g)$, i.e. the factor system $\rho$ is symmetric and condition (iii) holds.

If $\mathrm{char}(K)=p>0$ and $G$ contains an element $g$ of order $p$, then
\[
(1+v_g+v_g^2+\cdots+v_g^{p-1})^p=0
\]
and we get a contradiction. This implies that  condition (ii) also follows.

Assume that $G$ is hamiltonian and $\gp{g,h\mid g^4=h^4=1, g^2=h^2, g^h=g^{-1}}\cong Q_8$ is the quaternion group of order $8$. Then $h\m1gh=g\m1$ and $i=1$. Therefore in this case by (\ref{E:1}) we have  $u_h=v_gu_hv_g$, i.e.
\begin{equation}\label{E:2}
v_h=v_gv_hv_g,
\end{equation}
where $v_h=u_h\mu_h$ and $v_g^4=v_h^4=1$. Since $G$ contains $2$-elements,  it follows from  (ii) that $\mathrm{char}(K)\not=2$.

$K$ being  an algebraically closed field, it is clear that there exist nonzero elements $\alpha, \beta\in K$ for which $\alpha^2+ \beta^2=0$. Then by
(\ref{E:2}) it is easy to verify that
\[
w=\alpha(v_g^2v_h-v_h)+ \beta(v_g^3v_h-v_gv_h)
\]
is a nonzero nilpotent element of $K_\rho G$.

Indeed, $h\in \mathrm{Supp}\big(\alpha(v_g^2v_h-v_h)\big)$, but $h\not \in \mathrm{Supp}\big(\beta(v_g^3v_h-v_gv_h)\big)$.
Thus we have  $w\not=0$. Moreover, by (\ref{E:2}) we obtain that
$u_h^2v_g=v_gu_h^2$ and $u_hv_g^2=v_g^2u_h$.
Then $w^2=(v_g^2-1)^2(\alpha v_h+\beta v_gv_h)^2$. Since $(v_g^2-1)^2=2(1-v_g^2)$ and
\[
\begin{split}
(\alpha v_h+\beta v_gv_h)^2=(\alpha^2+\beta^2)v_h^2+\alpha\beta v_h^2(v_g^2+1)v_g&\\
=\alpha\beta v_h^2(v_g^2+1)v_g&,
\end{split}
\]
we obtain $w^2=2(1-v_g^2)\alpha\beta v_h^2(1+v_g^2)v_g=0$,
which is impossible. Hence  condition (i) follows, as requested.
\end{proof}


\section{Regular crossed products}

An associative ring $R$ with unity
is called {\it regular (strongly regular)} if for
every  $a\in R$ there is an element $b\in R$, such that $aba=a$
($ba^2=a$, respectively). A ring $R$ is called $\xi^*$-ring ($\xi^* N$-ring)   if for every
$a\in R$ there exists $b\in R$ such that  $aba-a$ is a central
(central nilpotent, respectively) element of $R$. It is clear that every regular ring is
a $\xi^* N$-ring and every $\xi^* N$-ring is a $\xi^*$-ring (see \cite{Mihovski, Utumi}).

By the theorem of Auslander, Connell and Willamayor (see \cite{Connell}, Theorem 3, p.660), it is well known
that a group ring is regular if and only if $K$ is regular, $G$ is a locally finite group
and the order of every element $g\in G$ is invertible in $K$.

\smallskip
Our first result for this section is the following.
\begin{theorem}\label{T:2}
Let $K*G$  be a crossed product of the group $G$ over the ring $K$ such that
one of the following conditions is satisfied:
\begin{itemize}
\item[(i)] $K*G$  is a $\xi^* N$-ring;
\item[(ii)] $K*G$ is  $n$-weakly regular.
\end{itemize}
Then $G$ is a torsion group.
\end{theorem}
\begin{proof}
(i) Suppose that $g\in G$ is an element of infinite order. Then there exists a  $b\in K*G$ and a natural number $n\geq 1$ such that
\[
x=(u_g-1)b(u_g-1)-(u_g-1)
\]
is a central element of $K*G$ and $x^n=0$. If $n=1$, then $x=0$ and
\[
(u_g-1)[b(u_g-1)-1]=0.
\]
Since,  by Corollary \ref{C:1},  the element $u_g-1$ is not a left zero divisor in $K*G$,
we obtain that $b(u_g-1)=1$, i.e. $u_g-1$ is a left invertible
element in $K*G$, which is also impossible. Therefore $n>1$ and
\[
x^n=(u_g-1)[b(u_g-1)-1]x^{n-1}=0.
\]
In the same way we obtain that $z_1=[b(u_g-1)-1]x^{n-1}=0$. Suppose that for some $k\geq 1$ we have\quad
$
z_k=[b(u_g-1)-1]^kx^{n-k}=0$.
\quad
If $1<k<n$,  as far as $x$ is central,
\[
\begin{split}
z_k&=x[b(u_g-1)-1]^kx^{n-k-1}\\
&=(u_g-1)[b(u_g-1)-1]^{k+1}x^{n-k-1}=0.
\end{split}
\]
Now applying Corollary \ref{C:1} we obtain that
\[
z_{k+1}=[b(u_g-1)-1]^{k+1}x^{n-k-1}=0.
\]
Thus, by induction we conclude that $z_n=[b(u_g-1)-1]^{n}=0$.

The last equality shows that there exists  $z\in K*G$ such that $z(u_g-1)=1$, which, by Corollary \ref{C:1}, is impossible.

(ii) Suppose that $g\in G$ is an element of infinite order. Then for some $b,c\in K*G$ we have  $u_g-1=(u_g-1)b(u_g-1)^nc$.
By Corollary \ref{C:1} we have
\[
(u_g-1)[1-b(u_g-1)^nc]=0,
\]
 we conclude that $b(u_g-1)^nc=1$. Hence it follows  that $b(u_g-1)x=1$, where $x=(u_g-1)^{n-1}c$. If
$e=xb(u_g-1)$, then
\[
e^2=x[b(u_g-1)x]b(u_g-1)=xb(u_g-1)=e,
\]
i.e. $e$ is a central idempotent of $K*G$. Thus we have
\[
\begin{split}
1=b(u_g-1)x&=b(u_g-1)[xb(u_g-1)]x\\
&=xb(u_g-1)[b(u_g-1)x]=xb(u_g-1),
\end{split}
\]
i.e. $u_g-1$ has a left invertible element $xb\in K*G$. Now again by Corollary \ref{C:1} we obtain a contradiction, so the proof is complete. \end{proof}

\begin{corollary}\label{C:2}
If the crossed product  $K*G$  is a  regular ring,  then $K$ is also a regular ring and  $G$ is a torsion group.
\end{corollary}
\begin{proof}
The claim follows from  Theorem \ref{T:2} and Lemma \ref{L:1}.
\end{proof}

Observe that the theorem of Auslander, Connell and Willamayor (see \cite{Connell}, Theorem 3, p.660) does  not apply  for crossed products.
Indeed,  if $K$ is a non-perfect field of characteristic $p>0$ and $G$ is the $p^\infty$-group, then there
exists a twisted group ring $K_\rho G$, which must be  a field (see \cite{Passman_radical_II}, Proposition 4.2).

If  $G$ satisfies the maximum condition for finite normal subgroups and the group ring $KG$ is a $\xi^* N$-ring,
then  $G$ is locally finite (see  \cite{Rakhnev}, Theorem 3, p.16).

We shall prove the locally finiteness of $G$ without the
assumption of the maximum condition when $K$ is a field. First we
recall that (see \cite{Passman_book}, p.308)
\[
\Delta(G)=\{g\in G\mid [G:C_G(g)]<\infty\}
\]
is a subgroup of $G$, where $C_G(g)$ is the centralizer of $g$ in $G$. Furthermore, we put
\[
\Delta^{p}(G)=\gp{\; g\in \Delta(G) \mid g \quad \text{is a $p$-element}\;  },
\]
that is   the subgroup of $\Delta(G)$ which  is generated by all $p$-elements of $\Delta(G)$.

Now we are ready to prove the following.
\begin{theorem}\label{T:3}
Let $KG$ be the group algebra of a group $G$ over a  field $K$.
If   $KG$  is a  $\xi^*N$-ring, then $G$ is a locally finite group.
Moreover,  if $\mathrm{char}(K)=p>0$   then $\Delta^p(G)$ contains all $p$-elements of $G$.
\end{theorem}

\begin{proof}
Let $\mathfrak{N}(KG)$ be the union of all  nilpotent ideals of $KG$.
In particular,   the central nilpotent elements of $KG$ are in $\mathfrak{N}(KG)$ and, consequently,    $KG/\mathfrak{N}(KG)$ is a regular ring.

Assume $\mathrm{char}(K)=p>0$. By Theorem 8.19 (\cite{Passman_book}, p.309),
\[
\mathfrak{N}(KG)=\mathfrak{Rad}(K[\Delta^p(G)])KG,
\]
where $\mathfrak{Rad}(K[\Delta^p(G)])$ is the Jacobson radical of the group ring $K[\Delta^p(G)]$.
Obviously, the augmentation ideal $\omega(K[\Delta^p(G)])$ is a
maximal  ideal of  $K[\Delta^p(G)]$, so
\[
\mathfrak{N}(KG)=\mathfrak{Rad}(K[\Delta^p(G)])KG\subseteq \omega(K[\Delta^p(G)])KG.
\]
It is well-known (see \cite{Connell}, Theorem 3, p.660) that
\[
K[G/(\Delta^p(G))]\cong KG/\omega(K[\Delta^p(G)])KG
\]
and therefore the group algebra $K[G/\Delta^p(G)]$ is regular, as a homomorphic
image of $KG/\mathfrak{N}(KG)$. This implies,  by the theorem  of Auslander, Connell and Willamayor
(see \cite{Connell}, Theorem 3, p.660), that $G/\Delta^p(G)$ is locally finite
and  has no $p$-element. Thus  we obtain that $\Delta^p(G)$ contains all the $p$-elements of $G$ and the group $G$ is locally finite (see \cite{Kargapolov_Merzliakov}, Theorem 23.1.1, p.215).

If  $\mathrm{char}(K)=0$, then $\mathfrak{N}(KG)=0$ and $KG$ is regular.
According to  Auslander-Connell-Villamayor's theorem the proof is complete.
\end{proof}


\section{$n$-weakly regular twisted group algebras}

Let $n\geq 2$ be a fixed natural number.
A ring $R$ is called {\it $n$-weakly regular} \cite{Gupta} if  for every  $a\in R$
there  exist elements $b,c\in R$ such that $a=aba^nc$.

Obviously, an $n$-weakly regular ring $R$ has no  nonzero nilpotent element. Indeed, if $R$
contains a nonzero nilpotent element, then there exists a nonzero nilpotent element $a\in R$ with $a^2=0$. Hence
$a=aba^nc=0$, which is impossible. From this fact we can conclude that all idempotents of an $n$-weakly regular ring are central.

In \cite{Bovdi_Langi} (Theorem 2, p.119) it was proved that the group algebra $KG$ over a  field $K$ is $n$-weakly regular ($n\geq 2$)
if and only if $K$ and $G$ satisfy at least one of the following two conditions:
\begin{itemize}
\item[(i)] $\mathrm{char}(K)=p>0$ and $G$ is an abelian torsion  group  without $p$-elements;
\item[(ii)] $\mathrm{char}(K)=0$ and $G$ is either an abelian  torsion group  or a hamiltonian
gr$G=Q\times E\times A$, where $A$ is an abelian  torsion group without $2$-elements and the equation
$x^2+y^2+z^2=0$ in $KA$    has only the trivial solution.
\end{itemize}
In the case when $K$ is an algebraically close field, this result
can be extended to.

\begin{theorem}\label{T:4}
A  twisted group algebra $K_\rho G$ of a group $G$ over the algebraically closed field $K$  is
$n$-weakly regular ($n\geq 2$) if and only if the following conditions hold:
\begin{itemize}
\item[(i)] $G$ is an abelian  torsion group;
\item[(ii)] the order of every  element of $G$ is invertible in $K$;
\item[(iii)]  the factor system $\rho$ is symmetric.
\end{itemize}
\end{theorem}

\begin{proof}
Suppose that  $K_\rho G$ is $n$-weakly regular. Then conditions (i), (ii) and (iii) hold by  Theorems \ref{T:1} and \ref{T:2}.

Conversely, if $K$ and $G$ satisfy the conditions (i), (ii) and (iii), then
$K_\rho G$ is a commutative ring. Let $a\in K_\rho G$ be an arbitrary element. Then $a\in K_\rho H$, where
$H=\langle  \mathrm{Supp}(a)\rangle$ is a finite abelian group. Since   $K_\rho H$ is a commutative semisimple artinian ring (\cite{Passman_radical_I}, Theorem 2.2), we conclude that $K_\rho H$ is a direct product of fields, so $K_\rho H$ is $n$-weakly regular. This implies that $K_\rho G$ is $n$-weakly regular, as requested.
\end{proof}

Analyzing the result of \cite{Bovdi_Langi} (see Theorem 2, p.119)  on  $n$-weakly regular group rings
and \cite{Mihovski} (see Corollary  2, p.70) about strongly regular group rings we deduce that when $K$ is
a field, then these two classes coincide.

In the case of  twisted group algebras over an algebraically closed basic field we have the following.
\begin{corollary} \label{C:3}
Let $K_\rho G$ be a twisted group algebra  of a group  $G$ over an algebraically closed field $K$.
The following statements are equivalent:
\begin{itemize}
\item[(i)] $K_\rho G$ is strongly regular;
\item[(ii)] $K_\rho G$ is $n$-weakly regular for every natural number
$n\geq 2$;
\item[(iii)] $K_\rho G$ is $n$-weakly regular for some natural number
$n\geq 2$;
\item[(iv)] $G$ is an abelian  torsion group, the order of every  element of $G$ is invertible in $K$ and the
factor system $\rho$ is symmetric.
\end{itemize}
\end{corollary}

\begin{proof}
Suppose that $K_\rho G$ is a strongly regular ring. If $a\in K_\rho G$ and $a=a^2b$, then
$a=aba$, because $K_\rho G$ does not contain nilpotent elements. Now by induction it follows that $a=ab^nc$ for some $c\in K_\rho G$ and for every natural number $n\geq 1$. So (i) implies (ii) and, obviously, (ii) implies (iii). By the preceding theorem, (iii) implies (iv). Finally, by the Auslander-Connell-Villamayor theorem and by  (iv) it follows that $K_\rho G$ is a commutative von Neumann ring and so (iv) implies (i). \end{proof}


\section{ $\xi N$-twisted group algebras}
A ring $R$ is called a $\xi N$-ring if for any $a\in
R$ there exists $b\in R$ such that  $a^2b-a$ is a  central
nilpotent element of $R$ (see \cite{Rakhnev}).

Obviously, every $\xi N$-ring is a $\xi$-ring  and, therefore,
(see \cite{Martindale}, Theorem 1, p.714) we deduce that every $\xi N$-ring is a $\xi^* N$-ring.
Moreover, (see \cite{Martindale}, Lemma 2, p.715) it follows that in  $\xi N$-rings all nilpotent
elements are central.

$\xi N$-group rings over commutative rings are described in \cite{Rakhnev} (Theorem 2, p.15).
From this description, it follows that a group ring $KG$
over a field $K$ of characteristic $p>0$ is a $\xi N$-ring
if and only if $G$ is an abelian  torsion group.

Finally  we prove the following.
\begin{theorem}\label{T:5}
A  twisted group algebra  $K_\rho G$ of a  group $G$ over the algebraically closed
field $K$ is a $\xi N$-ring  if
and only if the following conditions hold:
\begin{itemize}
\item[(i)] $G$ is an abelian  torsion  group;
\item[(ii)]  the factor system $\rho$ is symmetric.
\end{itemize}
\end{theorem}

\begin{proof}
Let $K_\rho G$ be a  $\xi N$-ring. Then  (\cite{Martindale}, Theorem 1, p.714) the ring  $K_\rho G$ is a  $\xi^* N$-ring and,
in view of Theorem \ref{T:2}, we conclude that $G$ is a torsion group.  As far as $K$ is an  algebraically
closed field, for every element $g\in G$ of order $n$ there exists an $\mu_g\in U(K)$, such that $v_g=u_g\mu_g$\quad
$(u_g\in U_G)$ and $v_g^n=1$. Then we put
\[
z=(v_g-1)v_h(1+v_g+v_g^2+\cdots +v_g^{n-1}), \qquad\quad (h\in H).
\]
Clearly, $z^2=0$ and therefore $z$ is a central element of $K_\rho G$. Thus $zv_h=v_hz$ and, so  we obtain the equality
\begin{equation}\label{E:3}
\begin{split}
2v_hv_gv_h+\sum_{i=1}^{n-1}v_g^iv_hv_gv_h&+\sum_{i=2}^{n-1}v_hv_g^iv_h\\
&=\sum_{i=1}^{n-1}v_g^iv_h^2+\sum_{i=0}^{n-1}v_hv_g^iv_hv_g.
\end{split}
\end{equation}
If $\mathrm{char}(K)=2$, then $2v_hv_gv_h=0$. Consequently for the product $v_hv_g^2v_h$
and for the corresponding supports we obtain the following three cases:
\begin{itemize}
\item[(\textsc{a}1)] $v_hv_g^2v_h=v_g^iv_hv_gv_h$, $hg^2h=g^ihgh$ and \quad   $hgh\m1=g^i$ \quad   ($1\leq i\leq n-1$);
\item[(\textsc{a}2)] $v_hv_g^2v_h=v_hv_g^iv_hv_g$, $hg^2h=hg^ihg$ and \quad   $hgh\m1=g^{2-i}$ \quad ($1\leq i\leq n-1$);
\item[(\textsc{a}3)] $v_hv_g^2v_h=v_g^iv_h$, $hg^2h=g^ih^2$ and \quad   $hg^2h\m1=g^i$ \quad ($1\leq i\leq n-1$).
\end{itemize}
This shows  that $\gp{g^2}$ is a normal cyclic subgroup of $G$.

If $g$ is a $2$-element of $G$, then $1+v_g$ is nilpotent and by Lemma 2 of \cite{Martindale}  we deduce that
$1+v_g$  is a central element of $K_\rho G$. Therefore $v_gv_h=v_hv_g$ for every $h\in G$.

If $g$ is an element of odd order, then $\gp{g^2}=\gp{g}$ and from (\textsc{a1}), (\textsc{a}2) and (\textsc{a}3) we obtain that every cyclic subgroup of $G$ is normal, i.e. $G$
is either abelian,  or hamiltonian. Since the $2$-elements of $G$ are central, we conclude that $G$ is an abelian  torsion group,
i.e.  condition (i) holds.  Now by (\textsc{a}1) and (\textsc{a}2) it follows that $i=1$ and $v_gv_h=v_hv_g$.
In case (\textsc{a}3) we have $i=2$ and $v_hv_g^2=v_g^2v_h$. But
$\gp{v_g^2}=\gp{v_h}$, so $v_h$ commutes with $v_g^i$ for all $i=1,\ldots,n-1$. Therefore  condition (ii) also holds.

Now, suppose that $\mathrm{char}(K)\not=2$. Then by (\ref{E:3}),  we conclude that for the product $v_hv_gv_h$ we have the following four cases:
\begin{itemize}
\item[(\textsc{b}1)] $v_hv_gv_h=v_g^iv_h^2$, $hgh=g^ih^2$ and \quad   $hgh\m1=g^i$ \quad ($1\leq i\leq n-1$);
\item[(\textsc{b}2)] $v_hv_gv_h=v_hv_g^iv_hv_g$, $hgh=hg^ihg$ and \quad   $hgh\m1=g^{1-i}$ \quad ($0\leq i\leq n-1$);
\item[(\textsc{b}3)] $v_hv_gv_h=-v_hv_g^iv_h$, $hgh=hg^ih$ and $g^{i-1}=1$, which is impossible, because $2\leq i\leq n-1$ and $g$ is of order $n$;
\item[(\textsc{b}4)] $v_hv_gv_h=-v_g^iv_hv_gv_h$, $hgh=g^ihgh$ and $g^{i}=1$, which is impossible, because $1\leq i\leq n-1$.
\end{itemize}
Therefore $\gp{g}$ is a normal cyclic subgroup of $G$ for every $g\in G$. Hence $G$  is either abelian or a hamiltonian group.

Assume that $G$ is hamiltonian and $\gp{g,h\mid g^4=1, h^2=g^2, hgh\m1=g\m1}\cong Q_8$. Then by (\textsc{b}1) and  (\textsc{b}2), it follows that either $i=3$ or $i=2$, respectively. Hence we obtain that $v_hv_g=v_g^3v_h$, where $v_g^4=v_h^4=1$.

Let $(\alpha,\beta)$ be a nontrivial solution of the equation $x^2+y^2=0$ in $K$. Then as in the proof of  Theorem \ref{T:1} we establish that
\[
w=\alpha (v_g^2v_h-v_h)+\beta (v_g^3v_h-v_gv_h)
\]
is a nonzero nilpotent element of $K_\rho  G$
with $z^2=0$. Therefore $w$ is a central element of $K_\rho  G$. But $wv_h\not=v_hw$, so we obtain a contradiction. Thus $G$ is abelian and  condition (i) holds. If $gh=hg$, then by (\textsc{b1}) and (\textsc{b}2)  it follows that either $i=1$ or $i=2$, respectively. Hence we obtain that $v_hv_g=v_gv_h$ for all $g,h\in G$ and so  condition (ii) also follows.

Conversely, if the conditions (i) and (ii) hold, then $K_\rho  G$ is a commutative ring. For every element $a\in K_\rho  G$ with $H=\langle  \mathrm{Supp}(a) \rangle$,
the ring $K_\rho H$ is artinian and $R\cong K_\rho H/\mathfrak{Nil}(K_\rho  H)$ is a finite sum of fields. Therefore $R$ is strongly regular  and hence
$K_\rho  H$ is a $\xi N$-ring. Since  $a\in K_\rho H$, we deduce that $K_\rho  G$ is a $\xi N$-ring.

Note that if $K_\rho G$ is a $\xi N$-ring, then the periodicity of $G$ can be proved directly. Indeed, if $g\in G$ is an element of infinite order and $z=(u_g-1)^2x-(u_g-1)$ is a central nilpotent element of
$K_\rho G$, then $z^n=0$ for some $n\geq 1$. By Corollary \ref{C:1}  we deduce that
$[(u_g-1)x-1]z^{n-1}=0$.

Using the fact that $z$ is central, we can prove  by induction that
\[
[(u_g-1)x-1]^kz^{n-k}=0
\]
for every $k\geq 1$. Therefore $[(u_g-1)x-1]^{n}=0$. This equality shows that  $u_g-1$ is right invertible in $K_\rho G$, which again  is impossible by Corollary \ref{C:1}.\end{proof}

\end{document}